\documentclass[10pt]{article}

\usepackage{amssymb} 
\usepackage{latexsym} 
\usepackage{amsmath}
\usepackage{mathtools} 
\usepackage{thmtools}
\usepackage{mathabx}
\usepackage{tabls} 
\usepackage{graphicx}
\usepackage{todonotes}
\usepackage{overpic}
\usepackage{subcaption}
\usepackage{tensor}
\usepackage{dutchcal}
\usepackage[english]{babel}
\usepackage{ntheorem}
\usepackage{a4wide}
\usepackage{caption}
\usepackage{tikz}
\usepackage[utf8]{inputenc}
\usepackage{color}
\usepackage{epsfig}
\usepackage{graphicx}
\usepackage[all]{xy}
\input xy
\usepackage[mathscr]{eucal} %use \mathsrc{X} to type eucal X
\usepackage{enumerate}
\usepackage{enumitem}
\usepackage{accents}
\usepackage{float}
\usetikzlibrary{graphs}

\newtheorem{theorem}{Theorem}

\newtheorem{remark}[theorem]{Remark}
\newtheorem{prop}[theorem]{Proposition}
\newtheorem{lemma}{Lemma}

\newtheorem{corollary}[theorem]{Corollary}

\newenvironment{proof}{\noindent{\bf Proof.}}%
{\hspace*{\fill}$\Box$\par\vspace{4mm}}

\newenvironment{theoremproof}[1]{\par\noindent\textbf{Proof of Theorem}\space\textbf{#1.}\!}{\hfill $\Box$\mybreak\noindent}

\newenvironment{lemmaproof}[1]{\par\noindent\textbf{Proof of Lemma}\space\textbf{#1.}\!}{\hfill $\Box$\mybreak\noindent}

\DeclareCaptionFormat{ruled}{
	#1#2#3\par\vspace{-.65\baselineskip}\hrulefill\par\vspace{-.83\baselineskip}\hrulefill}
\def\cadre{$$\vcenter\bgroup\advance\hsize by -2em\noindent
	\refstepcounter{equation}(\theequation)~\ignorespaces}
\makeatletter
\def\endcadre{\egroup\eqno$$\global\@ignoretrue}

\newcommand{\mybreak} {\par\vspace{2mm}\noindent}
\newcommand\mydfrac[2]{{\displaystyle\frac{#1}{#2}}}

\makeatletter
\def\imod#1{\allowbreak\mkern10mu({\operator@font mod}\,\,#1)}
\makeatother

\newcommand{\comment}[1]{}

\newcommand{\kn}[2] {#1^{\underline{#2}}}
\newcommand{\kns}[3]{#1_{#2}^{\kern0.1em\underline{#3}\phantom{#2}}}

\newcommand{\col} {\mathbf{c}}

\newcommand{\Pm}{\mathbb{P}}
\newcommand{\pra}[1] {\Pm\left(#1\right)}
\newcommand{\uno} {\mathsf{1}}

\newcommand{\E} {\mathbb{E}}
\newcommand{\V} {{\rm var}}
\newcommand{\Gu} {\mathbb{G}}

\newcommand{\G} {\mathcal{G}}

\newcommand{\Ea}[1] {\E\left(#1\right)}

\newcommand{\var}[1] {\V\left(#1\right)}
\newcommand{\size}[1] {\mathsf{e}(#1)}
\newcommand{\cut} {\mathsf{e}}

\newcommand{\Inj}{{\sf Inj}}

\newcommand{\range}{{\sf im}}

\usepackage{tikz}
\usetikzlibrary{shapes.geometric}
\tikzset{gon/.style={name=tmp,regular polygon,regular polygon sides=#1,minimum
		size=10pt,inner sep=0pt},
	polygon side/.style args={#1--#2}{
		insert path={(tmp.corner #1)-- (tmp.corner #2)}}}
\newcommand{\FlagGraph}[3][]{\ifnum#2=2%
	\tikz[baseline=(tmp1)]{\node[circle,inner sep=0.7pt,fill] (tmp1) at (0,0){};
		\node[#1,circle,inner sep=0.7pt,fill] (tmp2) at (0,10pt){};
		\ifx#3\empty%
		\else
		\draw[#1] (tmp1) -- (tmp2);
		\fi}
	\else%
	\tikz[baseline=(tmp.south)]{\node[#1,gon=#2]{};
		\foreach \X in {1,...,#2}{\fill (tmp.corner \X) circle (1pt);}
		\draw[#1,polygon side/.list={#3}]}
	\fi}

\setlist[itemize]{labelindent=\parindent, leftmargin=0.4cm}

\title{A universal dichotomy for concentration in randomly colored graphs.} 
\author{AA}
\date{}
\author{Nicola Apollonio\footnote{Consiglio Nazionale delle Ricerche, Istituto per le Applicazioni del
		Calcolo, M. Picone, Via dei Taurini 19, 00185 Roma, Italy
		\texttt{nicola.apollonio@cnr.it}}
}
\begin{document}

	\maketitle		
		\begin{abstract}
		We establish necessary and sufficient conditions for the concentration of counts of monochromatic and bichromatic edges in both randomly colored graphs and randomly colored random graphs. Let $\zeta$ be Euclidean norm of the degree sequence of a graph normalized by the graph size. We prove that when the vertices are randomly colored with $s$ colors such that the fraction of vertices in each color class is bounded away from zero, only two asymptotic regimes emerge. If $\zeta=o(1)$, then the sizes of the subgraphs induced by the color classes concentrate around their expected values. If $\zeta=\Theta(1)$, then concentration depends on the color balance: the total number $M$ of monochromatic edges concentrates if and only if the distribution of colors in the color classes becomes asymptotically uniform while for colorings with persistent imbalance, $M$ stays bounded away from its mean with positive probability. The same dichotomy holds for a broad class of randomly colored random graphs. 
	\end{abstract}

	\pagestyle{plain}
\section{Introduction}
Random colorings with prescribed class sizes arise in combinatorics and in applications such as network homophily and community detection. Let $G$ have $n$ vertices and $m$ edges, and let $\col=(c_1,\ldots,c_s)$ be a given composition of $n$ into $s$ parts with $2\leq s\leq n$. A $\col$-coloring of $V(G)$ is a map $f:V(G)\to[s]$ satisfying $|f^{-1}(i)|=c_i$ for every $i$. We equip the set of all $\col$-colorings of $V(G)$ with the uniform measure
$$\Pm_\col(f)=\mydfrac{c_1!\cdots c_s!}{n!},$$
and denote the resulting probability space by $\mathsf{R}_{\col,n}$. For distinct $i,j\in[s]$, let $M_i(G)$ count the edges with both endpoints of color $i$, and let $M_{i,j}(G)$ count those with endpoint colors $i$ and $j$. Thus
\begin{align}
	M_i(G)&=\sum_{uv\in E(G)}\uno\left(f(u)=f(v)=i\right)\label{eq:M_i}\\
	M_{i,j}(G)&=\sum_{uv\in E(G)}\left\{\uno(f(u)=i, f(v)=j)+\uno(f(u)=j,f(v)=i)\right\}\label{eq:M_{i,j}}
\end{align}
where $\uno(B)$ denotes the indicator of event $B$.
 
The random variable $M_i(G)$ is the size of a subgraph of $G$ induced by a randomly chosen set of $c_i$ vertices. It is an interesting combinatorial object in its own right. For instance, deciding whether the support of $M_i(G)$ contains $0$ is equivalent to deciding whether $G$ has an independent set with at least $c_i$ vertices and ${n\choose c_i}\pra{M_i(G)=0}$ is the $c_i$-th coefficient of the independence polynomial $I(G,x)$. The existence of proper vertex-colorings and clique partitions of $G$ can be read as events in $\mathsf{R}_{\col,n}$: $(M_1(G)=0,\ldots, M_s(G)=0)$ and $(M_1(G)={c_1\choose 2},\ldots, M_s(G)={c_1\choose 2})$, respectively. Moreover,  
$$\lim_{n\rightarrow \infty}\max\left\{\pra{M_1(G)={c_1\choose 2}-1} \ |\ \text{$G$ has $n$ vertices}\right\},$$
is the \emph{inducibility} of $K_{c_1}-e$ (\cite{pippago,LMR}) studied in depth in \cite{hirst,LMR} and the Celebrated Edge Statistic Theorem 	\cite{alonetal,KST,FoxSau, MMNT} gives a sharp universal upper bound on $\pra{M_1(G)=k}$ when the number of vertices of $G$ is sufficiently large.

The random vector $(M_1(G),\ldots,M_s(G))$ plays a role in the context of Network Homophily \cite{mac,newman}. In this context, one observes a network partitioned into distinguishable clusters and compares it with a random benchmark obtained by sampling without replacement. Assessing homophily amounts to performing a statistical test on the tails of appropriate scalar functions of the random vector $(M_1(G),\ldots,M_s(G))$ \cite{apofrasa1}.
 
In this paper, we study weak laws of large numbers for the random matrix $\boldsymbol{M}(G)$ defined by
\[
\boldsymbol{M}_{i,j}(G)=\begin{cases}
	M_i(G) & \text{if $i=j$}\\
	M_{i,j}(G) & \text{if $i\neq j$}
\end{cases}
\]
and the related statistics 
\begin{equation}\label{eq:M_and_L}
	M(G)=\sum_{i=1}^sM_i(G)\quad\text{and}\quad L(G)=m-M(G)
\end{equation}
the numbers of monochromatic and bichromatic edges, respectively. Note that 
$$M(G)+L(G)=m.$$
We show that the asymptotic behavior of these statistics is governed by the parameter 
$$\zeta(G)=\mydfrac{\sqrt{\Sigma_2(G)}}{m}$$ 
where $\Sigma_2(G)$ is the sum of squared degrees of $G$. As shown by the following theorem when $\zeta(G)=o(1)$, all edge counts concentrate around their expectations for every color distribution in which the fraction of the vertices in each color class is bounded away from zero. When $\zeta(G)=\Theta(1)$, the total monochromatic count concentrates exactly in the asymptotically balanced regime, while persistent imbalance yields anti-concentration. Here and throughout the rest of the paper we write $X_n \xrightarrow{P} X$ to mean that the sequence of random variables $(X_n)_{n\geq 1}$, tacitly defined on the same probability space, converges to $X$ in probability. Also, when we say that $(X_n)_{n\geq 1}$ (relatively) concentrates, we mean that $X_n/\Ea{X_n}\xrightarrow{P}1$. Denote by $\size{G}$ the number of edges of $G$.  
	\begin{theorem}\label{thm:main}
	Let $\left(G_n\right)_{n\geq 1}$ be a sequence of randomly $\col_n(s)$-colored graphs of order $n$ and size $\size{G_n}=m_n$ where $\col_n(s)=(c_{1,n},\ldots,c_{s,n})$ and $\sum_{i=1}^sc_{i,n}=n$. For $i=1,\ldots,s$ let $\gamma_{i,n}=c_{i,n}/n$ and   $$\boldsymbol{\gamma}_n=(\gamma_{1,n},\ldots,\gamma_{s,n})$$ be the \emph{color-distribution vector}. Also let $\zeta_n=\zeta(G_n)$, $\boldsymbol{\upsilon}_s=s^{-1}\mathbf{1}_s$ and $\widebar{\boldsymbol{M}}(G_n)$ be the matrix defined by
	\[
	\widebar{\boldsymbol{M}}_{i,j}(G_n)=\begin{cases}
		\Ea{M_i(G_n)} & \text{if $i=j$}\\
		\Ea{M_{i,j}(G_n)} & \text{if $i\neq j$}
	\end{cases}
	\]
	If 
	$$\liminf_n\min_{i\in [s]}\gamma_{i,n}>0\quad\text{and}\quad m_n\to\infty$$ 
	then the following hold.
	\begin{enumerate}[label={\rm (\roman*)}]
		\item\label{com:i} If $\zeta_n=o(1)$, then, for every $i,j\in[s]$,
		$$\frac{\boldsymbol{M}_{i,j}(G_n)}{\widebar{\boldsymbol{M}}_{i,j}(G_n)}\xrightarrow{P}1.$$
		Hence, in particular, both $L(G_n)$ and $M(G_n)$ concentrate. 
		\item\label{com:ii} If $\zeta_n=\Theta(1)$, then
		\[
		\frac{M(G_n)}{\Ea{M(G_n)}}\xrightarrow{P}1
		\quad\text{and}\quad
		\frac{L(G_n)}{\Ea{L(G_n)}}\xrightarrow{P}1
		\quad\Longleftrightarrow\quad
		\|\boldsymbol{\gamma}_n-\boldsymbol{\upsilon}_s\|\to 0.
		\]
		Moreover, if $\liminf_n\|\boldsymbol{\gamma}_n-\boldsymbol{\upsilon}_s\|>0$, then there exist constants $\epsilon,\delta>0$ such that
		\[
		\liminf_{n\to\infty}\pra{|M(G_n)-\Ea{M(G_n)}|>\epsilon m_n}\geq\delta.
		\]
		The same estimate holds for $L(G_n)$. The implication 
		\[
		\|\boldsymbol{\gamma}_n-\boldsymbol{\upsilon}_s\|\to 0\Longrightarrow\quad\frac{M(G_n)}{\Ea{M(G_n)}}\xrightarrow{P}1
		\quad\text{and}\quad
		\frac{L(G_n)}{\Ea{L(G_n)}}\xrightarrow{P}1
		\quad
		\]
		remains valid without the assumption $\zeta_n=\Theta(1)$.
	\end{enumerate}
\end{theorem}
The following Law of Large Numbers for balanced graph partitions in dense graphs is a straightforward, yet representative, consequence of Theorem \ref{thm:main} and Theorem \ref{thm:moments}. 
\begin{corollary}\label{coro:szmeredy_type}
	Let $(G_n)_{n\geq 1}$ be a sequence of graphs of order $n$, size $\size{G_n}=m_n$, and edge density
	$$\frac{2m_n}{n(n-1)}\to p>0.$$ 
	Fix $s\geq 2$ and let $f$ be any coloring sampled from $\mathsf{R}_{\col_n,n}$ where 
	$$c_{i,n}=n/s+o(n),\quad i=1,\ldots,s$$ 
	and let $\bigsqcup_{i=1}^s V_i$ be the partition induced by $f$. Then, for distinct $i,j\in[s]$,
	$$\frac{\size{G_n[V_i]}}{\binom{c_{i,n}}{2}}=p+o_P(1)\quad\text{and}\quad\frac{\cut_{G_n}(V_i,V_j)}{c_{i,n}c_{j,n}}=p+o_P(1)$$
	where $\cut_{G_n}(V_i,V_j)$ is the number of edges between $V_i$ and $V_j$. Thus every block and every cut has edge density $p+o_P(1)$ with high probability.
\end{corollary}
%	\begin{proof}
%		If $p>0$, then $m_n=\Theta(n^2)$ and
%		$$\zeta_n^2=\frac{\Sigma_2(G_n)}{m_n^2}\leq\frac{2(n-1)}{m_n}=o(1),$$
%		so the conclusion follows from Theorem \ref{thm:main}. 
%	\end{proof}
This results, can be stated more generally, even for sparser graphs. In this form, it shows that sampling a balanced partition yields, with high probability, induced subgraphs and cuts of density $p+o_P(1)$, as in the binomial random graph $\Gu(n,p)$.

In Theorem \ref{thm:main}, the assumption $\liminf_n\min_{i\in [s]}\gamma_{i,n}>0$ excludes degenerate color classes, while the assumption $m_n\to\infty$ excludes concentration caused only by a vanishing number of edges. Moreover, $\zeta_n=o(1)$ holds for bounded maximum degree or for $m_n=\omega(n)$, whereas $\zeta_n=\Theta(1)$ implies $\Delta(G_n)=\Theta(m_n)$ and $m_n=O(n)$ (see Remark \ref{rem:sum_of_squares_order}).

The heart of the proof of Theorem \ref{thm:main} is the exact computation of the second moment structure of the edge counts and the exact asymptotics of the leading coefficient of the asymptotic variance. Applying these results, along with some additional probabilistic tool (see Lemma \ref{lem:convpro}), shows that for a rather large class of random graphs, Theorem \ref{thm:main} lifts to random graphs in the following sense.
\begin{theorem}\label{thm:rdmain}
	Let $(\G_n)_{n\geq 1}$ be a sequence of random simple graphs on $[n]$. Independently of $\G_n$, sample a coloring from $\mathsf{R}_{\col_n(s),n}$, where the color-distribution vector $\boldsymbol{\gamma}_n=n^{-1}\col_n(s)$ is bounded away from both the boundary and the center of the standard $s$-dimensional simplex $\boldsymbol{\Delta}_s$. Let $\size{\G_n}$ be the size of $\G_n$, and define $\zeta(\G_n)=0$ on the event $\{\size{\G_n}=0\}$. Assume that
	\begin{enumerate}[label={\rm (*)}] 
		\item\label{com:ast} $\E(\size{\G_n})\to\infty$ and $\size{\G_n}/\Ea{\size{\G_n}}\to 1$ in quadratic mean.
	\end{enumerate}
	Then
	$$\frac{M(\G_n)}{\Ea{M(\G_n)}}\xrightarrow{P}1$$
	if and only if
	\begin{equation}\label{eq:eqpsuffcond} 
		\mydfrac{\Ea{\Sigma_2(\G_n)}}{\left[\Ea{\size{\G_n}}\right]^2}\to 0
	\end{equation}
	and, moreover, this condition is equivalent to
	\begin{equation}\label{eq:psuffcond} 
		\zeta(\G_n)\xrightarrow{P} 0.
	\end{equation}
\end{theorem}
 We note that assumption \ref{com:ast}, causes loss of generality. However, rewriting it as
$$\var{\size{\G_n}}=o\left(\left[\Ea{\size{\G_n}}\right]^2\right)$$
shows that it is satisfied by many models of random graphs (see Proposition \ref{prop:star} and the examples in Section \ref{sec:examples}). We note explicitly that Theorem \ref{thm:rdmain} recovers for random graph essentially the same dichotomy based on the asymptotic probabilistic behavior of the natural random counterpart of the invariant $\zeta_n$ described in \eqref{eq:psuffcond}. Fortunately, this dichotomy translates into the deterministic easily checkable condition \eqref{eq:eqpsuffcond}.

The rest of the paper is organized as follows: the remainder of he present section is devoted to some notation. In Section \ref{sec:mom} we compute the second moment structure of the statistics we are interested in Theorem \ref{thm:main}. In section \ref{sec:main}, we devise the leading term of the asymptotic variance of the statistics dealt with in Theorem \ref{thm:main} and prove the theorem. Finally, In Section \ref{sec:examples}, we illustrate Theorem \ref{thm:main} by deterministic graphs, illustrate Theorem \ref{thm:rdmain} by random graphs examples and prove it. The proof of both theorems relies on a probabilistic lemma based on the notion of uniform integrability. This lemma, along with other technical results, is proved in the Appendix.     

\paragraph{Preliminaries} A \emph{composition $\col$ of an integer $n$ into $s$ classes} is a tuple $(c_1,\ldots,c_s)$ such that $c_1+\cdots+c_s=n$. For a graph $G$ and $v\in V(G)$, $d_G(v)$ denotes the degree of $v$ in $G$. We use $\Sigma_2(G)$ for the sum of the squares of the degrees of $G$. For a positive integer $s$ and a non-negative integer $k$, the degree-$k$ elementary symmetric polynomial in $s$ indeterminates is an element of the polynomial ring $\mathbb{Z}[X_1,\ldots,X_s]$ defined by 
$$E_k(X_1,\ldots,X_s)=\sum_{\substack{A\subseteq [s]\\|A|=k}}\prod_{i\in A}X_i.$$ 
If $k=0$, then $E_0(X_1,\ldots,X_s)=1$ while if $k>s$ we set $E_k(X_1,\ldots,X_s)=0$. Also the $k$-th \emph{power sum} is the polynomial $P_k\in\mathbb{Z}[X_1\ldots,X_s]$ defined by 
$$P_k(X_1,\ldots,X_s)=\sum_{j=1}^sX_i^k.$$
In particular $P_0=s$, $P_1=E_1$. Moreover, Newton's identities, valid for every $k$ such that $1\leq k\leq s$, relate the elementary symmetric polynomial to the power sum by
$$kE_k=\sum_{j=1}^k(-1)^{j-1}E_{k-j}P_j.$$
The first three non-trivial such identities yield the system 
\begin{align*}
	&\begin{cases} 
		3E_3-E_2P_1+E_1P_2-P_3= 0 \\ 
		2E_2-E_1P_1+P_2 =0 \\ 
		E_1-P_1=0 
	\end{cases}
\end{align*}
For later purposes, the system above can be equivalently written as
\begin{equation}\label{eq:newtid}
	\begin{cases} 
		E_3-\dfrac{P_1^3-3P_1P_2+2P_3}{6}= 0 \\ 
		2E_2-P^2_1+P_2 =0 \\ 
		E_1-P_1=0 
	\end{cases}
\end{equation} 
We adhere to Knuth's notation for the falling factorial: for non-negative integers $a$ and $b$, $$\kn{a}{b}=a(a-1)\cdots(a-b+1),$$ 
with $\kn{a}{0}=1$ and $\kn{a}{b}=0$ whenever $b>a$.  
\section{The first two moments of the random sizes of subgraphs induced by colors}\label{sec:mom}  
Lemma \ref{lem:counting} below describes the probability of certain fundamental events in $\mathsf{R}_{\col,n}$ that push forward the distribution of the random variables appearing in Theorem \ref{thm:moments}. Its proof is given in he appendix.
\begin{lemma}\label{lem:counting}
	Let $A_1,\ldots,A_k$ be pairwise disjoint subsets of $V(G)$ with $1\leq k\leq s$. Let $a_j=|A_j|$ and set $a=\sum_ja_j$. Let $\Inj([k],[s])$ be the set of injections from $[k]$ into $[s]$. In the random coloring space $\mathsf{R}_{\col,n}$, define the events 
	\begin{itemize}
		\item[--] $\mathcal{E}(A_1,\ldots,A_k,\iota)$, $\iota\in \Inj([k],[s])$: all vertices of $A_j$ receive color $\iota(j)$ for $j\in [k]$.
		\item[--] $\mathcal{E}(A_1,\ldots,A_k)$: all vertices of $A_j$ receive the same color and these colors are distinct for different $j\in [k]$.
	\end{itemize}
	Then 
	$$\pra{\mathcal{E}(A_1,\ldots,A_k,\iota)}=\mydfrac{1}{\kn{n}{a}}\prod_{j\in [k]}\kns{c}{\iota(j)}{a_j}$$
	and    
	$$\pra{\mathcal{E}(A_1,\ldots,A_k)}=\mydfrac{1}{\kn{n}{a}}\sum_{\iota\in \Inj([k],[s])}\prod_{j\in [k]}\kns{c}{\iota(j)}{a_j}.$$
	In particular, writing $e_k$ for the value of $E_k$ in $\col$, i.e., $e_k=E_k(c_1,\ldots,c_s)$, if $k=1$, then, with $A=A_1$ and $\iota(1)=i$,  
	\begin{equation}\label{eq:monedge}
	\pra{\mathcal{E}(A,i)}=\mydfrac{\kns{c}{i}{a}}{\kn{n}{a}};
	\end{equation}
	if $a_1=2$, $a_2=\cdots=a_k=1$, hence $a=k+1$, then 
	\begin{equation}\label{eq:biedge}
	\pra{\mathcal{E}(A_1,\ldots,A_k)}=\mydfrac{(k-1)!}{\kn{n}{k+1}}\Big[(n-k)e_k-(k+1)e_{k+1}\Big];
	\end{equation}
	if $a_1=a_2\ldots=a_k=b$, then 
	\begin{equation}\label{eq:severlaedges}
	\pra{\mathcal{E}(A_1,\ldots,A_k)}=\mydfrac{k!}{\kn{n}{kb}}E_k(\kns{c}{1}{b}\ldots \kns{c}{s}{b}),
	\end{equation}
which, for $k=2$ and $b=2$ specializes to 
\begin{equation}\label{eq:severlaedgesbuttwo}
	\pra{\mathcal{E}(A_1,A_2)}=\mydfrac{2\left(e_2(e_2-e_1+1)-e_3(2e_1-3)+2e_4\right)}{\kn{n}{4}}.
\end{equation}
 Finally, since all of the events above depends on the sets $A_1,\ldots,A_k$ only through $k$ and their cardinalities $a_i$'s, these events are actually functions of the composition $(a_1,a_2,\ldots,a_k)$ of the integer $a$. Therefore, we write $\mathcal{E}(a_1,\ldots,a_k,\iota)$ and $\mathcal{E}(a_1,\ldots,a_k)$ for $\mathcal{E}(A_1,\ldots,A_k,\iota)$ and $\mathcal{E}(A_1,\ldots,A_k)$, respectively.  
\end{lemma}
The preceding lemma allows us to compute the first two moments of the distribution of the random size of the subgraphs of $G$ induced by the vertices of color $j$, $j\in [s]$, under a random $\col$-coloring as well as the first two moments of the random number of the edges whose ends receive different colors under random $\col$-colorings. 
\begin{theorem}\label{thm:moments}
Let $G$ be a graph of order $n$ and size $m$ and, for $s\geq 2$, let $\col=(c_1,\ldots,c_s)$ be a composition of $n$ in $s$ parts. Write $e_k$ for the value of $E_k$ in $\col$, i.e., $e_k=E_k(c_1,\ldots,c_s)$. Let $M_i(G)$, $M_{i,j}(G)$, $M(G)$ and $L(G)$ be the random variables defined by \eqref{eq:M_i}, \eqref{eq:M_{i,j}}, and \eqref{eq:M_and_L}, respectively. Recall that $M(G)+L(G)=m$, hence $M(G)$ and $L(G)$ have the same variance.
Then, for $i=1,\ldots,s$, the mean $\widebar{M}_i(G)$ and the variance $\sigma^2_i(G)$ of $M_i(G)$ are given by
 $$\widebar{M}_i(G)=m\mydfrac{\kns{c}{i}{2}}{\kn{n}{2}}$$
and
$$\sigma_i^2(G)=\mydfrac{\kns{c}{i}{3}(n-c_i)}{\kn{n}{4}}\Sigma_2(G)-\Big(\Big[\mydfrac{\kns{c}{i}{2}}{\kn{n}{2}}\Big]^2-\mydfrac{\kns{c}{i}{4}}{\kn{n}{4}}\Big)m^2+\Big(\mydfrac{\kns{c}{i}{2}}{\kn{n}{2}}-2\mydfrac{\kns{c}{i}{3}}{\kn{n}{3}}+\mydfrac{\kns{c}{i}{4}}{\kn{n}{4}}\Big)m.$$
Likewise, for any $i,\,j\in [s]$, $i\neq j$, the mean $\widebar{M}_{i,j}(G)$ and the variance $\sigma^2_{i,j}(G)$ of $M_{i,j}(G)$ are given by 
$$\widebar{M}_{ij}(G)=2m\mydfrac{c_ic_j}{\kn{n}{2}}$$
and
$$\sigma_{i,j}(G)=\left(\frac{c_i\kn{c_j}{2}+\kn{c_i}{2}c_j}{\kn{n}{3}}-4\frac{\kn{c_i}{2}\kn{c_j}{2}}{\kn{n}{4}}\right)\Sigma_2(G)-4\left(\left(\frac{c_ic_j}{\kn{n}{2}}\right)^2-\frac{\kn{c_i}{2}\kn{c_j}{2}}{\kn{n}{4}}\right)m^2+2\left(\frac{c_ic_j}{\kn{n}{2}}-\frac{c_i\kn{c_j}{2}+\kn{c_i}{2}c_j}{\kn{n}{3}}+2\frac{\kn{c_i}{2}\kn{c_j}{2}}{\kn{n}{4}}\right)m$$

The means of $\widebar{L}(G)$ and $\widebar{M}(G)$ are given by
\[
\widebar{L}=\Ea{L}=\mydfrac{2me_2}{\kn{n}{2}}\quad\text{and}\quad\widebar{M}=\Ea{M}=\mydfrac{m}{\kn{n}{2}}\sum_{i=1}^s\kns{c}{i}{2}\]
while their common variance  $\sigma^2$ is given by 
$$\sigma^2(G)=\left(a(\col)-b(\col)\right)\Sigma_2(G)+m^2\left(b(\col)-4\left[\mydfrac{e_2}{\kn{n}{2}}\right]^2\right)+m\left(\mydfrac{2e_2}{\kn{n}{2}}-2a(\col)+b(\col)\right)$$
where 
\begin{equation}\label{eq:ab}
	a(\col)=\mydfrac{e_2}{\kn{n}{2}}+\mydfrac{3e_3}{\kn{n}{3}}\quad\text{and}\quad b(\col)=\mydfrac{4}{\kn{n}{4}}\Big(e^2_2-(n-1)e_2-3e_3\Big).
\end{equation}
\end{theorem}     
\begin{proof}
Since the expression for the moments of $M_i(G)$ and $M_{i,j}(G)$ are given in \cite{ourssr}, we prove only the expression for the mean and the common variance of $M(G)$ and $L(G)$. However, with the help of Lemma \ref{lem:counting}, the same technique can be easily adapted to prove the first part of the theorem as well. For ease of notation we omit the explicit reference to the graph $G$ in each of the random variables defined in the statement of the theorem. Observe that $L=\sum_{uv\in E(G)}Y_{uv}$ where $Y_{uv}$ is the indicator of the event that $u$ and $v$ receive different colors while $M=\sum_iM_i$. The event that $u$ and $v$ receive different colors is an event of the form $\mathcal{E}(1,1)$ as described in Lemma \ref{lem:counting}. By \eqref{eq:severlaedges}, with $k=2$, and $b=1$, one has
\[
\Ea{Y_{uv}}=\mydfrac{2e_2}{\kn{n}{2}}
\] 
so that, by linearity of the expectation 
\[
\widebar{L}=\Ea{L}=\mydfrac{2me_2}{\kn{n}{2}}\quad\text{and}\quad\widebar{M}=\Ea{M}=\mydfrac{m}{\kn{n}{2}}\sum_{i=1}^s\kns{c}{i}{2}.
\]
Since $L$ and $M$ have the same variance, it suffices to compute the variance of $L$. Using the identity $X^2=X$ valid for a Bernoulli random variable and writing $uv\sim u'v'$ to mean that the edges $uv$ and $u'v'$ are adjacent in $G$, the definition of variance yields, 
\[
\begin{split}
\var{L}&=\Ea{L^2}-\left(\Ea{L}\right)^2\\
&=\Big[\widebar{L}-\widebar{L}^2\Big]+2\Big(\sum_{uv\sim u'v'}\Ea{Y_{uv}Y_{u'v'}}+\sum_{uv\not\sim u'v'}\Ea{Y_{u,v}Y_{u'v'}}\Big).
\end{split}
\]
Therefore, the computation of $\var{L}$ reduces to computing the expected value of the product $Y_{u,v}Y_{u',v'}$ in the cases $uv\sim u'v'$ and $uv\not\sim u'v'$. In both cases, the expected value equals the probability that both factors are equal to 1. Observe that in both cases any $\col$-coloring of $G$ favorable to either events, restricts to a proper vertex coloring of the subgraph of $G$ spanned by $uv$ and $u'v'$.  

Let us start by computing $\pra{Y_{u,v}=1,Y_{u',v'}=1}$ assuming $uv\sim u'v'$. Since $uv$ and $u'v'$ are adjacent, they span a subgraph $H$ isomorphic to $K_{1,2}$. Every proper vertex coloring of $H$ induces a (unordered) partition of $V(H)$ into color classes. There are two types of such partitions: one of type $(2,1)$ and one of type $(1,1,1)$. Here, each entry in these tuples, is the cardinality of the corresponding color class. Therefore the event $(Y_{u,v}=1,Y_{u',v'}=1)$ occurs precisely when either $\mathcal{E}(2,1)$ or $\mathcal{E}(1,1,1)$ occur. Therefore, by \eqref{eq:biedge} with $a_1=2$, $a_2=1$ and $a=3$, and \eqref{eq:severlaedges}, with $k=3$ and $b=1$ one has
$$\pra{Y_{u,v}=1,Y_{u',v'}=1}=\pra{\mathcal{E}(2,1)}+\pra{\mathcal{E}(1,1,1)}$$
so that, writing $e_k$ for $e_k(c_1\ldots,c_s)$.
\[
\pra{Y_{u,v}=1,Y_{u',v'}=1}=\mydfrac{(n-2)e_2-3e_3}{\kn{n}{3}}+\mydfrac{6e_3}{\kn{n}{3}}=\mydfrac{(n-2)e_2+3e_3}{\kn{n}{3}}.
\]
Let us compute $\pra{Y_{u,v}=1,Y_{u',v'}=1}$ assuming $uv\not\sim u'v'$. Since $uv$ and $u'v'$ are not adjacent, they span a subgraph $H$ isomorphic to $2K_2$. As above, since any $\col$-coloring of $G$ restricts to a proper vertex coloring of $H$, it follows that every $\col$-coloring of $G$ induces a (unordered) partition of $V(H)$ into color classes. There are three types of such partitions: two partitions of type $(2,2)$, four partitions of type $(2,1,1)$ and one partition of type $(1,1,1,1)$. Reasoning exactly as in the previous case yields
$$\pra{Y_{u,v}=1,Y_{u',v'}=1}=2\pra{\mathcal{E}(2,2)}+4\pra{\mathcal{E}(2,1,1)}+\pra{\mathcal{E}(1,1,1,1)}$$
so that, recalling \eqref{eq:biedge} and \eqref{eq:severlaedges} and using that $e_1=n$,
\[
\begin{split}
	\pra{Y_{u,v}=1,Y_{u',v'}=1}&=4\mydfrac{\big(e_2(e_2-e_1+1)-e_3(2e_1-3)+2e_4\big)}{\kn{n}{4}}+8\mydfrac{(e_1-3)e_3-4e_4}{\kn{n}{4}}+\mydfrac{4!e_4}{\kn{n}{4}}\\
	&=\mydfrac{4}{\kn{n}{4}}\Big(e_2(e_2-n+1)-e_3(2n-3)+2e_4+2(n-3)e_3-8e_4+3!e_4\Big)\\
	&=\mydfrac{4}{\kn{n}{4}}\Big(e_2(e_2-n+1)-3e_3\Big).\\
\end{split}
\]
Let $a(\col)=\pra{Y_{u,v}=1,Y_{u',v'}=1}$ for $uv\sim u'v'$, and $b(\col)=\pra{Y_{u,v}=1,Y_{u',v'}=1}$ for $uv\not\sim u'v'$. Since, as we proved, both $a(\col)$ and $b(\col)$ are constant independent of $uv$ and $u'v'$, it follows that
\[
\begin{split}
	\var{L}&=\Big[\mydfrac{2me_2}{\kn{n}{2}}-\Big(\mydfrac{2me_2}{\kn{n}{2}}\Big)^2\Big]+2\Big[\pi_3(G)a(\col)+\left({m\choose 2}-\pi_3(G)\right)b(\col)\Big]\\
	&=\Big[\mydfrac{2me_2}{\kn{n}{2}}-\Big(\mydfrac{2me_2}{\kn{n}{2}}\Big)^2\Big]+(a(\col)-b(\col))(\Sigma_2(G)-2m)+b(\col)(m^2-m)=\\
	&=(a(\col)-b(\col))\Sigma_2(G)+m^2\Big(b(\col)-4\Big[\mydfrac{e_2}{\kn{n}{2}}\Big]^2\Big)+m\Big(\mydfrac{2e_2}{\kn{n}{2}}-2a(\col)+b(\col)\Big)
\end{split}
\]
and the proof is completed.
\end{proof}

\begin{remark}\label{rem:sum_of_squares_order}
	By counting in two ways, for any graph $G_n$ with $n$ vertices, $m_n$ edges and maximum degree $\Delta_n$, it holds that 
	$$\Sigma_2(G_n)=\sum_{v\in V(G)}d^2_{G_n}(v)=\sum_{uv\in E(G_n)}(d_{G_n}(u)+d_{G_n}(v)).$$ 
	Since for $uv\in E(G_n)$ it holds that $2\leq d_{G_n}(u)+d_{G_n}(v)\leq 2\Delta_n$ and $\Delta_n\leq \min\{n-1,m_n\}$, this yields 
	$$2m_n\leq\Sigma_2(G_n)\leq2\Delta_nm_n\leq2(n-1)m_n.$$
	In particular, if $\Delta_n=O(1)$, then $\Sigma_2(G_n)=\Theta(m_n)$ and, provided $m_n\to\infty$, $\zeta_n^2=\Theta(m_n^{-1})=o(1)$. Also, $m_n=\omega(n)$ implies $\zeta_n^2\leq2(n-1)/m_n=o(1)$. Conversely, $\zeta_n=\Theta(1)$ implies
	$$\kappa m_n^2\leq\Sigma_2(G_n)\leq2\Delta_nm_n$$
	for some $\kappa>0$, and hence $\Delta_n=\Theta(m_n)$. Combining the lower bound $\kappa m_n^2\leq\Sigma_2(G_n)$ with $\Sigma_2(G_n)\leq2(n-1)m_n$ also gives $m_n=O(n)$. Thus only $O(1)$ vertices can have degree $\Theta(m_n)$, and the degree sequence is far from regular.
\end{remark}
\begin{remark}\label{rem:trivial cocentration}
	All statistics in Theorem \ref{thm:moments} lie in $[0,m_n]$. Moreover,
	$$\widebar{M}_i(G_n)=m_n\mydfrac{\kns{c}{i,n}{2}}{\kn{n}{2}},$$
	so Markov's inequality gives $M_i(G_n)\xrightarrow{P}0$ whenever $m_n\kns{c}{i,n}{2}=o(\kn{n}{2})$, independently of the structure of $G_n$. The lower bound on $c_{i,n}/n$ in Theorem \ref{thm:main} excludes this trivial concentration.
\end{remark}

	\section{Proof of Theorem \ref{thm:main}}\label{sec:main}
Before proving Theorem \ref{thm:main}, we collect some useful asymptotic estimates needed in its proof.  
\begin{lemma}\label{lem:asymtotics}
	Let $\col_n(s)=(c_{1,n},\ldots,c_{s,n})$ be compositions of $n$ into a fixed number $s$ of parts, $\boldsymbol{\gamma}_n=n^{-1}\col_n(s)$ be the corresponding color-distribution vector with $\boldsymbol{\gamma}_n=(\gamma_{1,n},\ldots,\gamma_{s,n})$, and let $\boldsymbol{\upsilon}_s=s^{-1}\boldsymbol{1}_s$ be tthe center of standard $s$-dimensional simplex. Write $a_n=a(\col_n(s))$, $b_n=b(\col_n(s))$, and $e_{k,n}=E_k(\col_n(s))$. If $\liminf_n\min_i\gamma_{i,n}>0$, then
	\begin{enumerate}[label={\rm (\alph*)}]
		\item\label{com:a} $\displaystyle\frac{\kns{c}{i}{\ell}}{\kn{n}{\ell}}=\gamma_{i,n}^{\ell}+O(n^{-1})$ for every fixed $\ell\geq1$ and $i\in[s]$;
		\item\label{com:a1} $\displaystyle\frac{\kns{c}{i}{2\ell}}{\kn{n}{2\ell}}-\left(\frac{\kns{c}{i}{\ell}}{\kn{n}{\ell}}\right)^2=O(n^{-1})$ for every fixed $\ell\geq1$ and $i\in[s]$;
		\item\label{com:b} %If $\liminf_{n\to\infty}\min_{in\in [s]}\gamma_{i,n}>0$, then  %$$\left(\mydfrac{2e_{2,n}}{\kn{n}{2}}-2a_n+b_n\right)_{n\geq 1}$$ 
		%is a bounded sequence, 
		$\Big(b_n-4\Big[\mydfrac{e_{2,n}}{\kn{n}{2}}\Big]^2\Big)=O\left(\mydfrac{1}{n}\right)$
		\item\label{com:c} %If $\liminf_{n\to\infty}\min_{in\in [s]}\gamma_{i,n}>0$, then
		$a_n-b_n=\rho_n+O\left(\mydfrac{1}{n}\right)$ where
		\[
		\rho_n=\begin{cases}
			o(1) & \text{if} \left\|\boldsymbol{\gamma}_n-\boldsymbol{\upsilon}_s\right\|^2=o(1),\\
			\Theta(1) &  \text{if} \left\|\boldsymbol{\gamma}_n-\boldsymbol{\upsilon}_s\right\|^2=\Theta(1).
		\end{cases}
		\]
	\end{enumerate}
\end{lemma}
\begin{proof} For fixed $k$ and $c=\Theta(n)$, 	$$\mydfrac{\kn{c}{k}}{\kn{n}{k}}=\left(\mydfrac{c}{n}\right)^k+O(n^{-1}).$$
	This proves \ref{com:a} and \ref{com:a1}; substituting the same expansion in \eqref{eq:ab} proves \ref{com:b}. For \ref{com:c}, let $\hat e_{k,n}=E_k(\boldsymbol{\gamma}_n)$. Since $n^k/\kn{n}{k}=1+O(n^{-1})$,
	$$\mydfrac{e_{k,n}}{\kn{n}{k}}=\hat{e}_{k,n}+O(1/n).$$
	Hence
	$$b_n=\mydfrac{4}{\kn{n}{4}}\Big(e^2_{2,n}-(n-1)e_{2,n}-3e_{3,n}\Big)=4\hat{e}^2_{2,n}+O(1/n).$$
	and 
	$$b_n-4\Big(\mydfrac{e_{2,n}}{\kn{n}{2}}\Big)^2=O\left(\mydfrac{1}{n}\right).$$
	Recall that 
	$$a_n=\mydfrac{e_{2,n}}{\kn{n}{2}}+\mydfrac{3e_{3,n}}{\kn{n}{3}}.$$ 
	Therefore
	$$a_n-b_n=\hat{e}_{2,n}-4\hat{e}^2_{2,n}+3\hat{e}_{3,n}+O(1/n).$$
	Let $p_{n,k}$ be the value at $\boldsymbol{\gamma}_n$ of the $k$-th power sum polynomial $P_k$. Over the standard $s$-dimensional simplex $\boldsymbol{\Delta}_s$, $\hat{e}_{1,n}=p_{1,n}=1$. Thus, plugging $\boldsymbol{\gamma}_n$ into system \eqref{eq:newtid} yields 
	\[   
	\begin{cases}
		\phantom{-}\hat{e}_{3,n}&\hspace{-.7em}=\mydfrac{1-3p_{2,n}+2p_{3,n}}{6}\\
		\phantom{.}2\hat{e}_{2,n}&\hspace{-.7em}=1-p_{2,n}. 
	\end{cases}
	\]
	Thus
	\[
	\begin{split}
		a_n-b_n&=\hat{e}_{2,n}-4\hat{e}^2_{2,n}+3\hat{e}_{3,n}+O(1/n)\\
		&=\mydfrac{1-p_{2,n}}{2}-4\Big(\mydfrac{1-p_{2,n}}{2}\Big)^2+3\mydfrac{1-3p_{2,n}+2p_{3,n}}{6}+O(1/n)\\
		&=p_{3,n}-p_{2,n}^2+O(1/n).
	\end{split}
	\]
	Observe now that by the Cauchy--Schwarz inequality $p_{3,n}-p_{2,n}^2\geq 0$. Indeed, since $\boldsymbol{\gamma}_n\in \boldsymbol{\Delta}_s$, it follows that
	\[
	1=\sum_{i=1}^s\gamma_{i,n}=\sum_{i=1}^s\left[\gamma^{\frac{1}{2}}_{i,n}\right]^2.
	\] 
	Therefore
	\[
	p_{3,n}=\sum_{i=1}^s\gamma^3_{i,n}=\sum_{i=1}^s\left[\gamma^{\frac{3}{2}}_{i,n}\right]^2\sum_{i=1}^s\left[\gamma^{\frac{1}{2}}_{i,n}\right]^2\geq \left[\sum_{i=1}^s\gamma^2_{i,n}\right]^2=p_{2,n}^2.
	\]
	Equality holds if and only if the vectors $(\gamma^{\frac{3}{2}}_{1,n},\ldots,\gamma^{\frac{3}{2}}_{s,n})$ and $(\gamma^{\frac{1}{2}}_{1,n},\ldots,\gamma^{\frac{1}{2}}_{s,n})$ are proportional, which means that the positive coordinates of $\boldsymbol{\gamma}_n$ are equal. Thus, over the full standard simplex, the zero set consists of the probability vectors that are uniform on their support. The hypothesis $\liminf_n\min_i\gamma_{i,n}>0$ allows us to choose $\delta>0$ such that, for all sufficiently large $n$, $\boldsymbol{\gamma}_n$ belongs to the compact set
	$$K_\delta=\left\{\boldsymbol{x}\in\boldsymbol{\Delta}_s\mid \min_i x_i\geq\delta\right\}.$$
	On $K_\delta$, the unique zero of $P_3-P_2^2$ is $\boldsymbol{\upsilon}_s$. By continuity, if $\left\|\boldsymbol{\gamma}_n-\boldsymbol{\upsilon}_s\right\|\to0$, then $\rho_n=p_{3,n}-p_{2,n}^2\to0$, while if $\left\|\boldsymbol{\gamma}_n-\boldsymbol{\upsilon}_s\right\|$ stays bounded away from zero, then so does $\rho_n$. The proof of the lemma is thus completed.
\end{proof}
We also need the following probabilistic lemma in the proof of both Theorem \ref{thm:main} and Theorem \ref{thm:rdmain}. Its proof is deferred to the Appendix.
\begin{lemma}\label{lem:convpro}
	Let $(X_n)_{n\geq 1}$ and $(Y_n)_{n\geq 1}$ be two sequences of random variables on a common probability space such that $0\leq X_n\leq\kappa Y_n$ almost surely for some constant $\kappa>0$. Define $X_n/Y_n=0$ on the event $\{Y_n=0\}$. If $Y_n/\E(Y_n)\xrightarrow{P}1$, then
	\begin{equation}\label{eq:uast}
		\mydfrac{\E(X_n)}{\E(Y_n)}\to0
		\Longleftrightarrow
		\mydfrac{X_n}{Y_n}\xrightarrow{P}0. \tag{*}
	\end{equation}
	If, in addition, $\E(X_n)\geq\tau\E(Y_n)$ for some constant $\tau>0$ and $Y_n/\E(Y_n)\to1$ in quadratic mean, then
	\begin{equation}\label{eq:dast}
		\mydfrac{X_n}{\E(X_n)}\xrightarrow{P}1
		\Longrightarrow
		\mydfrac{\var{X_n}}{[\E(X_n)]^2}\to0. \tag{**}
	\end{equation}
\end{lemma}
\begin{theoremproof}{\ref{thm:main}}
	The proof is by the second-moment method \cite{alonspencer}. Let $\widebar{M}_i(G_n)$ and $\sigma^2_i(G_n)$ be the mean and the variance of $M_i(G_n)$ as given by Theorem \ref{thm:moments} and, likewise, referring to the same theorem, let $\widebar{L}(G_n)$, $\widebar{M}(G_n)$ and $\sigma^2(G_n)$ be the first moments of $L(G_n)$ and $M(G_n)$, respectively, and their common variance. For ease of notation we write $\Sigma_{2,n}$ for $\Sigma_2(G_n)$. %Let $\boldsymbol{\gamma}_n=\left(\gamma_{1,n},\ldots,\gamma_{s,n}\right)=n^{-1}\col_n(s)$ be the color-distribution vector, and recall that $\boldsymbol{\upsilon}_s$ is the center $\boldsymbol{1}_s/s$ of $\boldsymbol{\Delta}_s$. 
	In view of Theorem \ref{thm:moments}, the assumptions 
	$$\liminf_n\min_i\gamma_{i,n}>0\quad\text{and}\quad m_n\to\infty$$
	guarantee that none of the first moments above is $o(1)$. Moreover, the first of such assumptions also grants that $\boldsymbol{\gamma}_n$ lies in the relative interior of $\boldsymbol{\Delta}_s$. In particular, the sequence $(\gamma_{i,n})_{n\geq 1}$ is bounded away from zero for every $i\in [s]$. To prove \ref{com:i}, it is enough to establish relative concentration for every $M_i(G_n)$ and every $M_{i,j}(G_n)$. We first prove that
	\[\mydfrac{M_i(G_n)}{\widebar{M}_i(G_n)}\xrightarrow[n\to\infty]{\phantom{-} p\phantom{-}} 1\qquad \forall i\in [s].
	\]
	By Chebyshev’s inequality, a sufficient condition for $M_i(G_n)/\widebar{M}_i(G_n)$ to converge to $1$ in probability is 
	\begin{equation}\label{eq:suffalon}
		\sigma^2_i(G_n)=o\left(\widebar{M}_i^2(G_n)\right).
	\end{equation}
	By Lemma \ref{lem:asymtotics}.\ref{com:a}, and the definition of asymptotic equivalence, one has 
	$$\mydfrac{\kns{c}{i}{2}}{\kn{n}{2}}\sim \gamma_{i,n}^2\quad\text{and}\quad \widebar{M}_i(G_n)=m_n\mydfrac{\kns{c}{i}{2}}{\kn{n}{2}}\sim m_n\gamma_{i,n}^2,$$
	where we used the fact that the sequence $(\gamma_{i,n})_{n\geq 1}$ is bounded away from zero as $n\to\infty$.
	Therefore,
	$$\widebar{M}_i^2(G_n)\sim m_n^2\gamma_{i,n}^4=\Theta(m_n^2)$$
	On the other hand, by Theorem \ref{thm:moments} and still by Lemma \ref{lem:asymtotics}.\ref{com:a} and Lemma \ref{lem:asymtotics}.\ref{com:a1}, it follows that  
	\[
	\begin{split}
		\sigma_i^2(G_n)&=\mydfrac{\kns{c}{i}{3}(n-c_i)}{\kn{n}{4}}\Sigma_{2,n}+m_n^2\Big[\mydfrac{\kns{c}{i}{4}}{\kn{n}{4}}-\Big(\mydfrac{\kns{c}{i}{2}}{\kn{n}{2}}\Big)^2\Big]\\
		&\quad+m_n\Big[\mydfrac{\kns{c}{i}{2}}{\kn{n}{2}}-2\mydfrac{\kns{c}{i}{3}(n-c_i)}{\kn{n}{4}}-\mydfrac{\kns{c}{i}{4}}{\kn{n}{4}}\Big]\\
		&=(\gamma_{i,n}^3(1-\gamma_{i,n})+O(n^{-1}))\Sigma_{2,n}+O(m_n^2/n)+O(m_n).
	\end{split}
	\] 
	so that, using that $\zeta_n=o(1)\Rightarrow \zeta^2_n=o(1)$, one has
	\[
	\mydfrac{\sigma_i^2(G_n)}{\widebar{M}_i^2(G_n)}=O\left(\mydfrac{\Sigma_{2,n}}{m_n^2}\right)=O(\zeta^2_n)=o(1).
	\]
	Therefore $\sigma_i^2(G_n)=o\left(\widebar{M}_i^2(G_n)\right)$. The very same argument proves that $M_{i,j}(G_n)$ concentrates around its expected value and the proof of \ref{com:i} is completed.
	
	We now prove \ref{com:ii}. Let $\sigma^2(G_n)$ denote the common variance of $M(G_n)$ and $L(G_n)$. By Theorem~\ref{thm:moments},
	$$\sigma^2(G_n)=(a_n-b_n)\Sigma_2(G_n)+m_n^2\Big(b_n-4\Big[\mydfrac{e_{2,n}}{\kn{n}{2}}\Big]^2\Big)+m_n\Big(\mydfrac{2e_{2,n}}{\kn{n}{2}}-2a_n+b_n\Big).$$
	Dividing by $m_n^2$ and applying Lemma \ref{lem:asymtotics}.\ref{com:b}--\ref{com:c} yields the following variance identity
	\[
	\begin{split}
		\mydfrac{\sigma^2(G_n)}{m^2_n}&=(a_n-b_n)\zeta_n^2+\Big(b_n-4\Big[\mydfrac{e_{2,n}}{\kn{n}{2}}\Big]^2\Big)+\mydfrac{1}{m_n}\Big(\mydfrac{2e_{2,n}}{\kn{n}{2}}-2a_n+b_n\Big)\\
		&=\rho_n\zeta_n^2+o(1),
	\end{split}
	\]
	where $\rho_n=P_3(\boldsymbol{\gamma}_n)-P_2(\boldsymbol{\gamma}_n)^2$. Since $\Sigma_2(G_n)\leq2m_n^2$, we have $\zeta_n^2\leq2$. Hence, if $\boldsymbol{\gamma}_n\to\boldsymbol{\upsilon}_s$, then $\rho_n\to0$ and therefore $\sigma^2(G_n)=o(m_n^2)$, without any further assumption on $\zeta_n$. Under the assumption that $\boldsymbol{\gamma}_n$ lies in the relative interior of the standard simplex, Theorem \ref{thm:moments} yields
	$$\widebar{M}(G_n)=\Theta(m_n)\quad\text{and}\quad\widebar{L}(G_n)=\Theta(m_n),$$
	so that Chebyshev's inequality proves relative concentration of both variables.
	
	Conversely, assume $\zeta_n=\Theta(1)$ and that either $M(G_n)$ or $L(G_n)$ concentrates relatively. Let $X_n\in\{M(G_n),L(G_n)\}$ denote the corresponding random variable and treat $Y_n=m_n$ as a deterministic random variable. Hence $0\leq X_n\leq Y_n$, $Y_n/\E(Y_n)=1$, and $\E(X_n)=\Theta(m_n)$. Consequently, $\E(X_n)\geq\tau\E(Y_n)$ for some $\tau>0$ and all sufficiently large $n$. Applying \eqref{eq:dast} in Lemma \ref{lem:convpro} yields
	$$\mydfrac{\var{X_n}}{[\E(X_n)]^2}\to 0.$$
	Since $\E(X_n)=\Theta(m_n)$, it follows that
	$$\mydfrac{\var{X_n}}{m_n^2}\to 0.$$
	The variance identity above and $\zeta_n^2=\Theta(1)$ imply $\rho_n\to0$. Lemma \ref{lem:asymtotics}.\ref{com:c} then implies $\boldsymbol{\gamma}_n\to\boldsymbol{\upsilon}_s$, otherwise there would be a subsequence along which $\|\boldsymbol{\gamma}_n-\boldsymbol{\upsilon}_s\|=\Theta(1)$ and hence $\rho_n=\Theta(1)$.
	
	Finally, suppose that $\liminf_n\|\boldsymbol{\gamma}_n-\boldsymbol{\upsilon}_s\|>0$. By Lemma \ref{lem:asymtotics}.\ref{com:c}, the assumption $\zeta_n=\Theta(1)$, and the variance identity above, there is a constant $c>0$ such that
	$$\sigma^2(G_n)\geq c m_n^2$$
	for all sufficiently large $n$. Let
	$$Z_n:=\mydfrac{|M(G_n)-\widebar{M}(G_n)|}{m_n}.$$
	Hence $0\leq Z_n\leq 1$, $0\leq Z^2_n\leq Z_n$, and
	$$\E(Z_n)\geq\E(Z_n^2)=\mydfrac{\sigma^2(G_n)}{m_n^2}\geq c.$$
	Since $\E(Z_n)\leq\epsilon+\pra{Z_n>\epsilon}$ for any $\epsilon$ such that $0\leq \epsilon \leq 1$, choosing $\epsilon=c/2$ yields
	$$\pra{|M(G_n)-\widebar{M}(G_n)|>\epsilon m_n}\geq c/2.$$
	The same is true for $L(G_n)$ because
	$$L(G_n)-\widebar{L}(G_n)=-\bigl(M(G_n)-\widebar{M}(G_n)\bigr).$$
	This proves \ref{com:ii}.
\end{theoremproof}
\section{Examples}\label{sec:examples} 
We first give two deterministic examples and then illustrate and prove Theorem \ref{thm:rdmain}.
\subsection{Regular graphs} If $G_n$ is regular, then $\Sigma_2(G_n)=4m_n^2/n$. Hence $\zeta_n^2=4/n$. Thus all the statistics in Theorem \ref{thm:main} concentrate, independently of the density of $G_n$. For $G_n=K_n$ they are deterministic: in particular, $L(G_n)$ is the number of edges of the complete multipartite graph with part sizes $c_{1,n},\ldots,c_{s,n}$.

\subsection{Stars} Let $G_n\cong K_{1,n}$. In this case $\zeta^2_n=1+\frac{1}{n-1}=\Theta(1)$. Fix $s\geq 2$ and let $\col_n(s)=(c_{1,n},\ldots,c_{s,n})$ be a composition of $n$ and let $\boldsymbol{\gamma}_n$ be the corresponding color distribution. Let $\xi_n$ be the non-leaf vertex of $G_n$ and observe that the number of $2$-chromatic edges in $G_n$ is uniquely determined by the color of the center $\xi_n$: if $\xi_n$ receives color $i$ under a random coloring $f$, then $L(G_n)=n-c_{i,n}=n(1-\gamma_{i,n})$. Therefore,
$$L(G_n)=n\sum_{i=1}^s(1-\gamma_{i,n})\uno\left(f\in \Phi_\col \mid f(\xi_n)=i\right).$$
If $c_{i,n}=n/s$, then $L(G_n)$ assumes the value $n(1-\frac{1}{s})$ with probability 1. Consequently, the random variable degenerates and coincides with its mean. We can thus suppose that $\boldsymbol{\gamma}_n$ remains bounded away from the uniform vector $\boldsymbol{\upsilon}_s$.
Since $\pra{\{f\in \Phi_\col \ |\ f(\xi_n)=i\}}=\gamma_{i,n}$, which can be easily checked directly or after resorting to Lemma \ref{lem:counting} using the event $\mathcal{E}(\{\xi_n\},\iota)$ with $\iota(i)=i$, it follows that
$$\widebar{L}(G_n)=\Ea{L(G_n)}=n(1-\|\boldsymbol{\gamma}_n\|^2)$$
and it can be checked that this expression specializes the one given in Theorem \ref{thm:moments}. The hypothesis $\liminf_{n\to\infty}\min_{i\in [s]}\frac{c_{i,n}}{n}$ in Theorem \ref{thm:main} grants that $\gamma_{i,n}>\epsilon$ for some $\epsilon>0$ and for every $i\in [s]$. Possibly by permuting the entries of $\boldsymbol{\gamma}_n$ we may suppose that $\gamma_{1,n}\leq\gamma_{2,n}\leq\cdots\leq\gamma_{s,n}$. Since $\boldsymbol{\gamma}_n\in \boldsymbol{\Delta}_s$, $\boldsymbol{\gamma}_n$ as a probability distribution on $[s]$. Therefore, using that 
$$\|\boldsymbol{\gamma}_n\|^2=\sum_{i=1}^s\gamma^2_{i,n}=\sum_{i=1}^s\gamma_{i,n}\gamma_{i,n}$$
yields
$$\gamma_{1,n}<\|\boldsymbol{\gamma}_n\|^2<\gamma_{s,n}$$
where the strict inequalities are implied by the assumption that $\boldsymbol{\gamma}_n\neq\boldsymbol{\upsilon}_s$. Let  
$$\delta_n=\min\left\{\|\boldsymbol{\gamma}_n\|^2-\gamma_{1,n},\gamma_{s,n}-\|\boldsymbol{\gamma}_n\|^2\right\}$$
and, for $\delta>0$, let
$$A_\delta=\left\{i\in [s] \ |\ |n(1-\gamma_{i,n})-n(1-\|\boldsymbol{\gamma}_n\|^2)|> \delta\right\}.$$
Observe that $\{1,s\}\subseteq A_\delta$ for every $\delta$ such that $0<\delta<\delta_n$. Therefore    
\[
\begin{split}
\pra{|L(G_n)-\widebar{L}(G_n)|\geq \delta}=\pra{L(G_n)\in A_\delta}\geq \gamma_1+\gamma_s>0
\end{split}
\]
and we conclude that $L(G_n)$ remains bounded away from its mean with positive probability. 
\subsection{Random Graphs} 
Now let the underlying simple graph $\G_n$ be random and colored independently according to $\mathsf{R}_{\col_n,n}$. Then $M_i(\G_n)$, $M(\G_n)=\sum_{i=1}^sM_i(\G_n)$, $L(\G_n)=\size{\G_n}-M(\G_n)$, and $\zeta(\G_n)$ are random variables. Theorem \ref{thm:rdmain} recovers the deterministic dichotomy under the following sufficient condition for concentration of $\size{\G_n}$.
\begin{prop}\label{prop:star}
	Let $(\G_n)_{n\geq 1}$ be a sequence of random graphs. Assume that $\size{\G_n}$ can be written as the sum of $\Omega(n)$ independent, possibly non-identically distributed, random variables with finite positive means $\mu_j$ and finite variances $\tau_j$. If the $\tau_j$ are $O(1)$ and the arithmetic mean of the $\mu_j$ remains bounded away from zero, then assumption \ref{com:ast} is satisfied.
\end{prop}
\begin{proof}
	Let $\size{\G_n}=\sum_{j=1}^{N_n}X_j$ where $N_n=\Omega(n)$, the $X_j$s are independent, and $X_j$ has mean $\mu_j>0$ and finite variance $\tau_j$. Hence $\E(\size{\G_n})=\sum_j\mu_j$ and $\var{\size{\G_n}}=\sum_j\tau_j$. Let $\widebar{\mu}=N_n^{-1}\sum_j\mu_j$ and $\tau^*=\max_{1\leq j\leq N_n}\tau_j$. Since $\widebar{\mu}>C$ for some positive constant $C$, it follows that $\E(\size{\G_n})=\Omega(N_n)$. Moreover,
	$$\E\left(\mydfrac{\size{\G_n}}{\E(\size{\G_n})}-1\right)^2=\mydfrac{\var{\size{\G_n}}}{[\E(\size{\G_n})]^2}\leq \mydfrac{1}{N_n}\mydfrac{\tau^*}{\widebar{\mu}^2}=O\left(\mydfrac{1}{N_n}\right).$$
	We therefore conclude that $\E(\size{\G_n})\to\infty$ and $\size{\G_n}/\E(\size{\G_n})\to 1$ in quadratic mean, namely, assumption \ref{com:ast} is fulfilled.
\end{proof}
Let us now examine concrete instances of random graphs.
\begin{itemize}
	\item[--] Let $\G_n\sim\G_{n,p_n}$ with $n^2p_n\to\infty$. Then assumption \ref{com:ast} holds because $\size{\G_n}\sim \operatorname{Bin}(\binom{n}{2},p_n)$. Moreover,
	$$[\Ea{\size{\G_n}}]^2=\Theta(n^4p_n^2),\qquad \Ea{\Sigma_2(\G_n)}=O(n^2p_n+n^3p_n^2),$$
	so the ratio in \eqref{eq:eqpsuffcond} in Theorem \ref{thm:rdmain} is $O((n^2p_n)^{-1}+n^{-1})=o(1)$. Hence $M(\G_n)$ concentrates; in particular this includes $np_n\to\lambda>0$. Equivalently, conditional on the color classes, $M_i(\G_n)\sim \operatorname{Bin}({c_i\choose 2},p)$ and $M(\G)=\sum_i^sM_i(\G_n)$.
	\item[--] Let $\G_n$ be a random geometric graph on the two-dimensional torus with fixed radius $r$. Let $r$ be chosen below the injective radius, namely small enough to grant that balls of radius $r$ on the torus are isometric to Euclidean balls of the same radius. Write $p$ for the area of a ball of radius $r$. Hence $p=\pi r^2$. 
	Let $X_1,\ldots,X_n$ be the independent uniform points that generate the graph, set $\mathcal{e}_{ij}=\uno(d(X_i,X_j)\leq r)$. Indicators corresponding to vertex-disjoint unordered pairs of indices are independent. If two distinct pairs of indices share an element, say $ij$ and $ik$, then, conditionally on $X_i$,
	\[
	\pra{\mathcal{e}_{ij}=\mathcal{e}_{ik}=1\mid X_i}
	=\pra{X_j\in B(X_i,r)\mid X_i}\pra{X_k\in B(X_i,r)\mid X_i}=p^2,
	\]
	because every ball of radius $r$ has area $p$. Thus any two distinct edge indicators are independent, although the whole family is not jointly independent. Consequently,
	\[
	\Ea{\size{\G_n}}=\binom{n}{2}p,
	\qquad
	\var{\size{\G_n}}=\binom{n}{2}p(1-p).
	\]
	Moreover, conditionally on $X_i$, the indicators $(\mathcal{e}_{ij})_{j\neq i}$ are independent Bernoulli random variables with parameter $p$, which does not depend on $X_i$. Hence
	\[
	d_{\G_n}(i)\sim\operatorname{Bin}(n-1,p)
	\quad\text{and}\quad
	\Ea{\Sigma_2(\G_n)}
	=n\bigl((n-1)p(1-p)+(n-1)^2p^2\bigr).
	\]
	For fixed $p>0$, assumption \ref{com:ast} is fulfilled and the ratio in \eqref{eq:eqpsuffcond} is $O(n^{-1})$. Therefore Theorem \ref{thm:rdmain} implies concentration.
	\item[--] More generally, for an independent-edge model satisfying assumption \ref{com:ast}, Theorem \ref{thm:rdmain} shows that $M(\G_n)$ fails to concentrate whenever
	$$\liminf_{n\to\infty}\mydfrac{\Ea{\Sigma_2(\G_n)}}{[\Ea{\size{\G_n}}]^2}>0.$$
	This behavior occurs in a Chung--Lu model (see \cite{chunglu}) with weights $w_1=n$ and $w_i=1$ for $i\geq2$. Writing $W_n=\sum_kw_k=2n-1$, the edge probabilities are
	$$p_{i,j}=\mydfrac{w_iw_j}{W_n}=\begin{cases} 1/2+O(1/n) & \text{if $1\in\{i,j\}$},\\1/(2n)+O(1/n^2) & \text{otherwise}.
	\end{cases}$$
	Here $\Ea{\size{\G_n}}=\Theta(n)$ and $\var{\size{\G_n}}=O(n)$, so assumption \ref{com:ast} holds. Moreover, the degree of vertex $1$ has second moment $\Theta(n^2)$, and hence the ratio in \eqref{eq:eqpsuffcond} is bounded away from zero.
\end{itemize}
We now prove Theorem \ref{thm:rdmain}, using Lemma \ref{lem:convpro}.
\mybreak
\begin{theoremproof}{\ref{thm:rdmain}}
	Set
\[
q_n:=\mydfrac{1}{\kn{n}{2}}\sum_{i=1}^s\kns{c}{i}{2},\quad
A_n:=a_n-b_n,\quad
B_n=b_n-4\left(\mydfrac{e_{2,n}}{\kn{n}{2}}\right)^2,\quad
\text{and}\quad 
C_n=\mydfrac{2e_{2,n}}{\kn{n}{2}}-2a_n+b_n.
\]
Conditionally on $\G_n$, Theorem \ref{thm:moments} gives the identities
\[
\Ea{M(\G_n)\mid\G_n}=q_n\size{\G_n}
\]
and
\[
\var{M(\G_n)\mid\G_n}=A_n\Sigma_2(\G_n)+B_n(\size{\G_n})^2+C_n\size{\G_n}.
\]
Because the color-distribution vectors stay away from the boundary and from $\boldsymbol{\upsilon}_s$, Lemma \ref{lem:asymtotics} implies that $q_n=\Theta(1)$, $A_n=\rho_n+O(n^{-1})$ with $\rho_n$ bounded from above and away from zero, $B_n=O(n^{-1})$, and $C_n=O(1)$. Let $\mu_n=\Ea{\size{\G_n}}$. The laws of total expectation and total variance yield
\[
\Ea{M(\G_n)}=q_n\mu_n
\]
and
\begin{equation}\label{eq:random-total-variance}
	\begin{split}
		\var{M(\G_n)}={}&A_n\Ea{\Sigma_2(\G_n)}+B_n\Ea{(\size{\G_n})^2}+C_n\mu_n+q_n^2\var{\size{\G_n}}.
	\end{split}
\end{equation}
Assumption \ref{com:ast} gives
\[
\frac{\var{\size{\G_n}}}{\mu_n^2}\to0,
\qquad
\frac{\Ea{(\size{\G_n})^2}}{\mu_n^2}\to1,
\qquad
\frac{\mu_n}{\mu_n^2}\to0.
\]
If \eqref{eq:eqpsuffcond} holds, dividing \eqref{eq:random-total-variance} by $\mu_n^2$ yields
\[
\frac{\var{M(\G_n)}}{[\Ea{M(\G_n)}]^2}\to0,
\]
and Chebyshev's inequality proves concentration.

Conversely, suppose that
\[
\frac{M(\G_n)}{\Ea{M(\G_n)}}\xrightarrow{P}1.
\]
Since $0\leq M(\G_n)\leq\size{\G_n}$ and $\Ea{M(\G_n)}=q_n\mu_n$ with $q_n$ bounded away from zero, the second part of Lemma \ref{lem:convpro}, applied with $X_n=M(\G_n)$ and $Y_n=\size{\G_n}$, gives
\[
\frac{\var{M(\G_n)}}{[\Ea{M(\G_n)}]^2}\to0.
\]
Therefore $\var{M(\G_n)}/\mu_n^2\to0$. Isolating the term containing $\Ea{\Sigma_2(\G_n)}$ in \eqref{eq:random-total-variance}, and using the preceding estimates together with the fact that $A_n$ is bounded away from zero, proves \eqref{eq:eqpsuffcond}.

It remains to prove the equivalence with \eqref{eq:psuffcond}. Assumption \ref{com:ast} also implies
\[
\frac{(\size{\G_n})^2}{\Ea{(\size{\G_n})^2}}\xrightarrow{P}1
\quad\text{and}\quad
\frac{\Ea{(\size{\G_n})^2}}{\mu_n^2}\to1.
\]
Moreover, $\Sigma_2(\G_n)\leq2(\size{\G_n})^2$ almost surely. The first part of Lemma \ref{lem:convpro}, applied with $X_n=\Sigma_2(\G_n)$ and $Y_n=(\size{\G_n})^2$, therefore shows that \eqref{eq:eqpsuffcond} is equivalent to
\[
\frac{\Sigma_2(\G_n)}{(\size{\G_n})^2}=\zeta^2(\G_n)\xrightarrow{P}0.
\]
By the continuous mapping theorem, this is equivalent to \eqref{eq:psuffcond}. This completes the proof.
\end{theoremproof}
\section{Concluding remarks}
in this paper, we have established necessary and sufficient conditions for the concentration of counts of monochromatic and bichromatic edges in both randomly colored graphs and randomly colored random graphs. The resulting criterion for concentration depends on the parameter $\zeta(G)=\sqrt{\Sigma_2(G)}/\size{G}$. This problem is motivated by the study of induced subgraph statistics and by the practical need to assess homophily and community structures in network science. A natural extension of this work would be to consider different random coloring mechanisms and, ultimately, to allow for dependence between the random graph and its coloring.

\section*{Appendix}
\subsection{Proof of Lemma \ref{lem:counting}}
For a $\col$-coloring $f\in \Phi_\col$ and a set $U\subseteq V(G)$, we denote by $f(U)$ the set $\{f(u) \ |\ u\in U\}$, namely the set of distinct colors assumed by the vertices in $U$. We preliminarily need the following useful lemma about the symmetric elementary polynomials. 
\begin{lemma}\label{lem:elemetarysymm}
	Let $k$ and $s$ be positive integers and $E_k\in \mathbb{Z}[X_1,\ldots,X_s]$ be the elementary symmetric polynomial in $s$ indeterminates of degree $k$. Let $\widetilde{E}_k=E_k(X_1^2,\ldots,X_s^2)$. Then, the following identities hold in $\mathbb{Z}[X_1\ldots,X_s]$:
	$$\sum_{|I|=k}\left(\sum_{i\in I}X_i-k\right)\prod_{j\in I}X_j=(E_1-k)E_k-(k+1)E_{k+1},$$
	where $I$ ranges over the $k$-subsets of $[s]$, and
	$$\widetilde{E}_2=E_2^2-2E_1E_3+2E_4.$$
\end{lemma}
\begin{proof}
	For $I\subseteq [s]$ with $|I|=k$ set 
	$$S_I:=S_I(X_1\ldots,X_s)=\sum_{i\in I}X_i^2\prod_{j\in I\setminus\{i\}}X_j.$$ 
	Expanding the polynomial on the left hand side of the stated identity gives
	$$\sum_{|I|=k}\left(\sum_{i\in I}X_i-k\right)\prod_{j\in I}X_j=\sum_{|I|=k}S_I-kE_k.$$
	Multiply $E_k$ by $E_1$ and observe that the monomials in the product $E_1E_k$ are of two types: either square-free $k+1$-tuples, or $k$-tuples in which exactly one indeterminate is squared. The same square-free $k+1$-tuple appears in exactly $k+1$ products of the form $X_i\prod_{j\neq i}X_j$ while the monomials of the second type are precisely the $S_I$. Therefore  
	$$E_1E_k=(k+1)E_{k+1}+\sum_{|I|=k}S_I$$
	and subtracting $kE_k+(k+1)E_{k+1}$ from both side yields the first identity. For the second identity let 
	$$P=\{\{i,j\},\{h,k\}\subseteq [s] \mid |\{i,j\}\cap\{h,k\}|=1\}$$ 
	and 
	$$Q=\{\{i,j\},\{h,k\}\subseteq [s] \mid |\{i,j\}\cap\{h,k\}|=0\}.$$ 
	Expanding $E_2^2$ yields:
	\[\begin{split}
		E_2^2&=\left(\sum_{i<j}X_iX_j\right)^2=\left(\sum_{i<j}X_iX_j\right)\left(\sum_{h<k}X_hX_k\right)\\
		&=\sum_{i<j}X^2_iX^2_j+2\sum_PX_iX_jX_hX_k+2\sum_QX_iX_jX_hX_k\\
		&=\widetilde{E}_2+2\hspace{-0.6em}\sum_{i<j<h}\hspace{-0.5em}X_iX_jX_h(X_i+X_j+X_h)+2\cdot 3\hspace{-1.1em}\sum_{i<j<h<k}\hspace{-0.8em}X_iX_jX_hX_k,
	\end{split}\]
	where the coefficient $3$ in the last term accounts for the three partitions of a $4$-set into two disjoint $2$-sets. Next, by the first identity of the lemma,
	\[
	\begin{split}
		\sum_{i<j<h}\hspace{-0.5em}X_iX_jX_h(X_i+X_j+X_h)&=\sum_{i<j<h}\hspace{-0.5em}X_iX_jX_h(X_i+X_j+X_h\pm3)\\
		&=(E_1-3)E_3-4E_4+3E_3\\
		&=E_1E_3-4E_4.
	\end{split}
	\]
	Thus 
	$$\widetilde{E}_2=E_2^2-2(E_1E_3-4E_4)-6E_4=E_2^2-2E_1E_3+2E_4$$
	which is the desired identity.
\end{proof}   

\begin{lemmaproof}{\ref{lem:counting}}
Let us first prove the special cases \eqref{eq:monedge}, \eqref{eq:biedge} and \eqref{eq:severlaedges} assuming the first two formulas. If $k=1$, let $A_1=A$ and $a_1=a$. Also let $\iota(1)=i$. Then 
	$$\mydfrac{1}{\kn{n}{a}}\prod_{j\in [k]}\kns{c}{\iota(j)}{a_j}=\mydfrac{\kns{c}{i}{a}}{\kn{n}{a}},$$
	namely, \eqref{eq:monedge}. Let us prove \eqref{eq:biedge}. Let $I$ be any $k$-subset of $[s]$ and let $\bar{j}\in I$ be fixed. Then, for every $\iota\in \Inj([k],[s])$ such that $\iota(1)=\bar{j}$, one has
	$$\prod_{j\in I\setminus\{\bar{j}\}}c_j=\prod_{j\in [k]\setminus \{1\}}c_{\iota(j)}.$$
	Hence, 	
	\[\begin{split}
		\sum_{\iota\in \Inj([k],[s])}\prod_{j\in [k]}\kns{c}{\iota(j)}{a_j}&=(k-1)!\sum_{|I|=k}\sum_{i\in I}(c_i-1)\prod_{j\in I}c_j\\
		&=(k-1)!\sum_{|I|=k}\left(\sum_{i\in I}c_i-k\right)\prod_{j\in I}c_j.
	\end{split}\]
	Thus, by the first identity of Lemma \ref{lem:elemetarysymm},
	$$\sum_{\iota\in \Inj([k],[s])}\prod_{j\in [k]}\kns{c}{\iota(j)}{a_j}=(k-1)!\Big[(e_1-k)e_k-(k+1)e_{k+1}\Big].$$
	Since $e_1=E_1(c_1,\ldots,c_s)=\sum_{j=1}^sc_j=n$, \eqref{eq:biedge} follows. Let us prove \eqref{eq:severlaedges}. Observe that, since $a_j=b$ for every $j\in [k]$, it follows that 
	$$\sum_{\iota\in \Inj([k],[s])}\prod_{j\in [k]}\kns{c}{\iota(j)}{a_j}=k!\sum_{B\subseteq [s], \#B=k}\prod_{j\in B}\kns{c}{j}{b}.$$
	Indeed, for each $B\subseteq [s]$ such that $|B|=k$ there are precisely $k!$ injections $\iota: [k]\to [s]$ such that $\range{\iota}=B$. The latter relation, yields \eqref{eq:severlaedges} upon dividing by $\kn{n}{a}$ and recalling the definition of degree $k$ elementary symmetric polynomial. For \eqref{eq:severlaedgesbuttwo}, expand
	$$\sum_{i<j}\kns{c}{i}{2}\kns{c}{j}{2}=\sum_{i<j}c_i^2c_j^2-\sum_{i<j}(c_i^2c_j+c_ic_j^2)+e_2.$$
	Lemma \ref{lem:elemetarysymm} gives
	$$\sum_{i<j}c_i^2c_j^2=e_2^2-2e_1e_3+2e_4,\qquad
	\sum_{i<j}(c_i^2c_j+c_ic_j^2)=e_1e_2-3e_3.$$
	Substituting and using $e_1=n$, yields \eqref{eq:severlaedgesbuttwo}.
	
	It remains to prove the two general probability formulas.
	For $\iota\in \Inj([k],[s])$, let $B(\iota)$ be the set
	$$B(\iota)=\left\{f\in \Phi_\col \ |\ f(v)=\iota(j),\,\forall v\in A_j,\,j=1,\ldots,k\right\}.$$
	A $\col$-coloring is favorable to the event $\mathcal{E}(A_1,\ldots,A_k,\iota)$ precisely when $f\in B(\iota)$. For $h=1,\ldots,s$, let 
	\[
	\tilde{c}_h(\iota)=\begin{cases}
		c_h &  \text{if $h\not\in \range{\iota}$}\\
		c_h-a_{\iota^{-1}(h)} & \text{if $h\in \range{\iota}$}. 
	\end{cases} 
	\]
	If $a_j>c_{\iota(j)}$ for some $j$, then $B(\iota)=\varnothing$. Otherwise $|B(\iota)|$ is the number of colorings of the remaining $n-a$ vertices with residual class sizes $\tilde c_h(\iota)$, namely
	$$|B(\iota)|=\frac{(n-a)!}{\tilde{c}_1(\iota)!\cdots\tilde{c}_s(\iota)!}.$$
	Since
	$$\mydfrac{c_{\iota(j)}!}{\tilde{c}_{\iota(j)}(\iota)!}=\kns{c}{\iota(j)}{a_j}\qquad(j\in[k]),$$
	we obtain
	\[
	\begin{split}
		\pra{\mathcal{E}(A_1,\ldots,A_k,\iota)}&=\left.|B(\iota)|\middle/\mydfrac{n!}{c_1!\ldots c_s!}\right.\\
		&=\mydfrac{(n-a)!}{n!}\mydfrac{c_1!\ldots c_s!}{\tilde{c}_1(\iota)!\cdots\tilde{c}_s(\iota)!}\\
		&=\mydfrac{1}{\kn{n}{a}}\prod_{j\in [k]}\kns{c}{\iota(j)}{a_j}	
	\end{split}
	\]
	as stated. Notice that the formula holds even if $a_j>c_{\iota(j)}$ for some $j\in [k]$ because, in that case, $\kns{c}{\iota(j)}{a_j}=0$. Next observe that  
	$$\mathcal{E}(A_1,\ldots,A_k)=\bigcup_{\iota\in \Inj([k],[s])}\mathcal{E}(A_1,\ldots,A_k,\iota)$$
	and the union is disjoint. Hence
	\[
	\pra{\mathcal{E}(A_1,\ldots,A_k)}=\sum_{\iota\in \Inj([k],[s])}\pra{\mathcal{E}(A_1,\ldots,A_k,\iota)}
	=\mydfrac{1}{\kn{n}{a}}\sum_{\iota\in \Inj([k],[s])}\prod_{j\in [k]}\kns{c}{\iota(j)}{a_j},
	\]
	which is precisely the formula stated in the lemma.  
\end{lemmaproof}
\subsection*{Proof of Lemma \ref{lem:convpro}}
For the foregoing notions we refer to \cite{billy,kalle}, especially Chapter 3 in \cite{kalle}. Recall that a sequence $(Y_n)_{n\geq 1}$ of random variables is uniformly integrable (see \cite{billy}) if for all $\epsilon>0$ there exists some $\ell>0$ such that:
$$\sup_n\E(|Y_n|\mathbf{1}_{\{|Y_n|>\ell\}})<\epsilon$$
which intuitively means that the mass of the variables dissipates on their tails and generalizes to sequences (and more generally to families) of random variables, the behavior of a single random variable with finite expectation: it is in fact known that $X\in L^1$ if and only if $\lim_{\ell\to\infty}\E(|X|\mathbf{1}_{|X|\geq\ell})=0$. i.e., for integrable random variables, far tails contribute negligibly to the expectation. Uniformly integrable sequences are thus those for which the magnitude of this contribution can be controlled uniformly over all elements. We need the following two properties of uniformly integrable sequences of random variables, the first of which follows by the very same definition of uniform integrability and the second one is a combination of Lemma 3.11 and Proposition 3.12 in \cite{kalle}.
\begin{lemma}~~
	\begin{enumerate}[label={\rm (U.\arabic*)}]
		\item\label{com:u1} Let $(X_n)_{n\geq 1}$, $(Y_n)_{n\geq 1}$ and $(Z_n)_{n\geq 1}$ be sequences of random variables such that $0\leq X_n\leq Y_n$ and with $(Z_n)_{n\geq 1}$ uniformly integrable. If $(Y_n)_{n\geq 1}$ is uniformly integrable so are $(X_n)_{n\geq 1}$ and $(aY_n+bZ_n)$ for every non-negative constant $a$ and $b$.
		\item\label{com:u2} Let $p\ge2$ be an integer and $X\in L^p$ a random variable. Furthermore let $(X_n)_{n\geq 1}$ be a sequence of random variables such that $X_n\in L^p$ with $X_n\xrightarrow{P} X$. Then the following statements are equivalent
		\begin{enumerate}[label={\rm (\alph*)}]
			\item\label{com:u2a} $X_n\to X$ in $L^p$;
			\item\label{com:u2b} $\|X_n\|_p\to\|X\|_p$;
			\item\label{com:u2c} $(|X_n|^p)_{n\geq 1}$ is uniformly integrable.
		\end{enumerate}
		Conversely, the first of the three statements implies that $X_n\xrightarrow{P} X$. Moreover, if $X_n\to X$ in distribution and $(X_n)_{n\geq 1}$ is uniformly integrable, then $\E(X_n)\to \E(X)$.
	\end{enumerate}	
\end{lemma} 

\begin{lemmaproof}{\ref{lem:convpro}} Assume that the $\E(X_n)/\E(Y_n)$ converges to zero and let $\epsilon>0$. We want to show that $\pra{X_n/Y_n>\epsilon} \to 0$ as $n\to\infty$. Fix $\delta\in (0,1)$ and observe that
	$$Y_n<(1-\delta)\E(Y_n)\Longrightarrow \left|\frac{Y_n}{\E(Y_n)}-1\right|>\delta.$$
	Therefore
	\[
	\begin{split}
		\pra{\mydfrac{X_n}{Y_n}>\epsilon}&=\pra{\frac{X_n}{Y_n}>\epsilon\bigwedge Y_n\geq(1-\delta)\Ea{Y_n}}+\pra{\frac{X_n}{Y_n}>\epsilon\bigwedge Y_n<(1-\delta)\Ea{Y_n}}\\
		&\leq \pra{\frac{X_n}{Y_n}>\epsilon\bigwedge Y_n\geq(1-\delta)\Ea{Y_n}}+\pra{\left| \mydfrac{Y_n}{\Ea{Y_n}}-1\right|>\delta}\\
		&\leq \pra{\frac{X_n}{(1-\delta)\Ea{Y_n}}>\epsilon}+\pra{\left| \mydfrac{Y_n}{\Ea{Y_n}}-1\right|>\delta}\\
		&\leq \frac{\Ea{X_n}}{\epsilon(1-\delta)\Ea{Y_n}}+\pra{\left| \mydfrac{Y_n}{\Ea{Y_n}}-1\right|>\delta}, 
	\end{split}
	\]
	where, in the last inequality, we used Markov's inequality. Since both the addenda converge to zero---the first one by the assumption on the ratio of the expected values and the second one by the definition of convergence in probability, it follows that $X_n/Y_n$ converges in probability to zero, as required. Let us assume now $X_n/Y_n$ converges to zero in probability. Hence, for any $\epsilon>0$, $\pra{X_n/Y_n>\epsilon}\to 0$ as $n\to\infty$. We want to prove that $\Ea{X_n}/\Ea{Y_n}$ converges to zero as $n\to\infty$. Observe that 
	\[
	\begin{split}
		\mydfrac{\Ea{X_n}}{\Ea{Y_n}}&=\Ea{\mydfrac{X_n}{\Ea{Y_n}}}=\Ea{\mydfrac{X_n}{Y_n}\mydfrac{Y_n}{\Ea{Y_n}}}
	\end{split}
	\] 
	Let 
	$$U_n=\mydfrac{X_n}{Y_n}\mydfrac{Y_n}{\Ea{Y_n}}\quad\text{and}\quad W_n=\mydfrac{Y_n}{\Ea{Y_n}}.$$ 
	Since $X_n/Y_n\xrightarrow{P} 0$, $Y_n/\Ea{Y_n}\xrightarrow{P} 1$, it follows that $U_n\xrightarrow{P} 0$ by Slutsky's Lemma. Hence, to prove the thesis it suffices to show that $\Ea{U_n}\to 0$. To this end observe that the definition of $W_n$ and the assumption $0\leq X_n\leq \kappa Y_n$ imply
	$$\Ea{W_n}=1\quad\text{and}\quad0\leq U_n\leq \kappa W_n.$$
	Since $W_n\to 1$ in quadratic mean, it follows by \ref{com:u2} (with $p=2$) that $(W_n)_{n\geq 1}$ is uniformly integrable. On the other hand, by \ref{com:u1}, $(U_n)_{n\geq 1}$ is dominated by the uniformly integrable sequence $(\kappa W_n)_{n\geq 1}$. Hence, still by \ref{com:u1}, $(U_n)_{n\geq 1}$ is uniformly integrable. Moreover, since $U_n\xrightarrow{P} 0$ it follows that $\Ea{U_n}\to 0$ by the last part of \ref{com:u2}. This completes the proof of the first statement of the lemma. 
	
	To prove the remaining part let $T_n=X_n/\Ea{X_n}$ and assume we have proved that $(T^2_n)_{n\geq 1}$ is uniformly integrable. Since $T_n\geq 0$ it follows that $(T_n-1)^2\leq T^2_n+1$. Clearly, the constant sequence $(1)_{n\geq 1}$ is uniformly integrable. Hence $\left((T_n-1)^2\right)_{n\geq 1}$ is uniformly integrable by \ref{com:u1}. By the continuous mapping Theorem 
	$$T_n\xrightarrow{P} 1\Longrightarrow (T_n-1)^2\xrightarrow{P} 0.$$
	Since $(T_n-1)^2\xrightarrow{P} 0$ and $\left((T_n-1)^2\right)_{n\geq 1}$ is uniformly integrable, by the last part of \ref{com:u2}, using that convergence in probability implies convergence in distribution, it follows that the expected value of $(T_n-1)^2$ converges to the expected value of its limit in probability so that
	$$\mydfrac{\var{X_n}}{[\E(X_n)]^2}=\E\left(\mydfrac{X_n}{[\E(X_n)]^2}-1\right)^2=\E\left((T_n-1)^2\right)\to 0.$$
	This completes the proof of the lemma once we show that $(T^2_n)_{n\geq 1}$ is uniformly integrable. However this follows by \ref{com:u1} because 
	$$T^2_n=\left(\mydfrac{X_n}{\E(X_n)}\right)^2\leq \mydfrac{\kappa^2}{\tau^2}\left(\mydfrac{Y_n}{\E(Y_n)}\right)^2=\mydfrac{\kappa^2}{\tau^2}W^2_n$$
	and $(W^2_n)_{n\geq 1}$ is uniformly integrable because $W_n$ converges to 1 in quadratic mean, namely $W_n\to 1$ in $L^2$, and since $1$ is in $L^2$, the implication \ref{com:u2a}$\Rightarrow$\ref{com:u2c} applies. The proof is thus completed.  
\end{lemmaproof} 

\end{document}